\pdfoutput=1

\DeclareSymbolFont{AMSb}{U}{msb}{m}{n}
\documentclass[11pt,a4paper,leqno,noamsfonts]{amsart}

   \makeatletter
   \renewcommand\@biblabel[1]{#1.}
   \makeatother
   \usepackage{multirow}

\usepackage{graphicx}
\usepackage{subcaption}
\usepackage{pythontex}

\usepackage{listings}

\usepackage[bookmarks=false]{hyperref}

\usepackage{listings}
\usepackage{color}
\definecolor{dkgreen}{rgb}{0,0.6,0}
\definecolor{gray}{rgb}{0.5,0.5,0.5}
\definecolor{mauve}{rgb}{0.58,0,0.82}

\usepackage{mathtools}

\usepackage[english]{babel}
\usepackage[dvipsnames]{xcolor}
\usepackage{graphicx,pifont,soul,physics} 
\usepackage[utopia]{mathdesign}
\usepackage[a4paper,top=3cm,bottom=3cm,
           left=3cm,right=3cm,marginparwidth=60pt]{geometry}

\usepackage{fancyhdr}
\usepackage{float}

\usepackage[utf8]{inputenc}
\usepackage{braket,caption,comment,mathtools,stmaryrd}
\usepackage[usestackEOL]{stackengine}
\usepackage{multirow,booktabs,microtype,relsize}

\usepackage[bookmarks=false]{hyperref}%
      \hypersetup{colorlinks,%
           citecolor=britishracinggreen,%
            filecolor=black,%
            linkcolor=cobalt,%
            urlcolor=cornellred}
      \numberwithin{equation}{section}
\usepackage[capitalise]{cleveref}
\usepackage[colorinlistoftodos]{todonotes}

    \usepackage[bookmarks=false]{hyperref}
\definecolor{antiquewhite}{rgb}{0.98, 0.92, 0.84}
\definecolor{buff}{rgb}{0.94, 0.86, 0.51}
\definecolor{palecopper}{rgb}{0.85, 0.54, 0.4}
\definecolor{fluorescentyellow}{rgb}{0.8, 1.0, 0.0}
\definecolor{bole}{rgb}{0.47, 0.27, 0.23}

\usepackage{amsmath,float}
\usetikzlibrary{decorations.markings}

\usepackage[vcentermath]{youngtab}
\input xy
\xyoption{all}

\usepackage{pgfkeys}
\usepackage{pgfopts}
\usepackage{ytableau}

\definecolor{cornellred}{rgb}{0.7, 0.11, 0.11}
\definecolor{britishracinggreen}{rgb}{0.0, 0.26, 0.15}
\definecolor{cobalt}{rgb}{0.0, 0.28, 0.67}
\DeclareSymbolFont{usualmathcal}{OMS}{cmsy}{m}{n}
\DeclareSymbolFontAlphabet{\mathcal}{usualmathcal}

\newcommand{\BA}{{\mathbb{A}}}

\newcommand{\BC}{{\mathbb{C}}}

\newcommand{\BN}{{\mathbb{N}}}

\newcommand{\BP}{{\mathbb{P}}}

\newcommand{\BZ}{{\mathbb{Z}}}

\newcommand{\CH}{{\mathcal H}}

\newcommand{\CP}{{\mathcal P}}

\newcommand{\Fm}{{\mathfrak{m}}}

\newcommand{\Fy}{{\mathfrak{y}}}
\newcommand{\Fz}{{\mathfrak{z}}}

\DeclareMathOperator{\Hilb}{Hilb}

\DeclareFontFamily{OT1}{rsfs}{}
\DeclareFontShape{OT1}{rsfs}{n}{it}{<-> rsfs10}{}
\DeclareMathAlphabet{\curly}{OT1}{rsfs}{n}{it}
\renewcommand\hom{\mathscr{H}\kern-0.3em\mathit{om}}

\newcommand\Hom{\operatorname{Hom}}

\DeclareMathOperator{\lHom}{\mathscr{H}\kern-0.3em\mathit{om}}
\DeclareMathOperator{\RRlHom}{\mathbf{R}\kern-0.025em\mathscr{H}\kern-0.3em\mathit{om}}

\DeclareMathOperator{\lExt}{{\mathscr{E}\kern-0.2em\mathit{xt}}}



\usepackage[all]{xy}
\usepackage{tikz}
\usepackage{tikz-cd}
\usepackage{adjustbox}
\usepackage{rotating}
\usepackage{comment}

\usetikzlibrary{matrix,shapes,intersections,arrows,decorations.pathmorphing}
\tikzset{commutative diagrams/arrow style=math font}
\tikzset{commutative diagrams/.cd,
mysymbol/.style={start anchor=center,end anchor=center,draw=none}}

\tikzset{
shift up/.style={
to path={([yshift=#1]\tikztostart.east) -- ([yshift=#1]\tikztotarget.west) \tikztonodes}
}
}

\theoremstyle{definition}

\newtheorem*{lemma*}{Lemma}
\newtheorem*{theorem*}{Theorem}
\newtheorem*{example*}{Example}
\newtheorem*{fact*}{Fact}
\newtheorem*{notation*}{Notation}
\newtheorem*{definition*}{Definition}
\newtheorem*{prop*}{Proposition}
\newtheorem*{remark*}{Remark}
\newtheorem*{corollary*}{Corollary}

\newtheorem*{conventions*}{Conventions}

\newtheorem{definition}{Definition}[section]

\newtheorem{example}[definition]{Example}

\newtheorem{remark}[definition]{Remark}

\newtheoremstyle{thm} 
        {3mm}
        {3mm}
        {\slshape}
        {0mm}
        {\bfseries}
        {.}
        {1mm}
        {}
\theoremstyle{thm}

\newtheorem{corollary}[definition]{Corollary}
\newtheorem{lemma}[definition]{Lemma}

\newtheorem{thm}{Theorem}

\newtheorem{cor}{Corollary}

\newtheoremstyle{ex} 
        {3mm}
        {3mm}
        {}
        {0mm}
        {\scshape}
        {.}
        {1mm}
        {}
\theoremstyle{ex}

\newtheoremstyle{sol} 
        {3mm}
        {3mm}
        {}
        {0mm}
        {\scshape}
        {.}
        {1mm}
        {}
\theoremstyle{sol}

\usepackage{tikz}
\usepackage{xparse}

\newcount\tableauRow
\newcount\tableauCol

\newenvironment{Tableau}[1]{%
  \tikzpicture[scale=0.5,draw/.append style={thick,black},
                      baseline=(current bounding box.center)]
    \tableauRow=-1.5
    \foreach \Row in {#1} {
       \tableauCol=0.5
       \foreach\k in \Row {
         \draw[thin](\the\tableauCol,\the\tableauRow)rectangle++(1,1);
         \draw[thin](\the\tableauCol,\the\tableauRow)+(0.5,0.5)node{$\k$};
         \global\advance\tableauCol by 1
       }
       \global\advance\tableauRow by -1
    }
}{\endtikzpicture}

\newtheorem*{Acknowledgments*}{Acknowledgments}
\DeclareMathAlphabet\BCal{OMS}{cmsy}{b}{n}

\address{Centre for Mathematical Sciences, University of Cambridge, Wilberforce Road, CB3 0WA, Cambridge, United Kingdom}

\title[A proof of the 1978  Brian\c{c}on-Iarrobino Conjecture in three dimensions]{A proof of the 1978  Brian\c{c}on-Iarrobino Conjecture in three dimensions}

\author{Owen Mackenzie}
\email{om379@cam.ac.uk}

\author{Fatemeh Rezaee}
\email{fr414@cam.ac.uk}

\begin{document}
\maketitle
\begin{abstract}

We resolve the Brian\c{c}on-Iarrobino Conjecture regarding the maximum singularity of $\mathcal{H}=\mathrm{Hilb}^{l}(\mathbb{A}^3)$, where $l$ is a tetrahedral number, by refining the work of Ramkumar-Sammartano in \cite{Ramkumar-Sammartano}. This also immediately implies the conjectural necessary condition for a point of $\mathcal{H}$ to have the maximal singularity, suggested by the second-named author in \cite{Rezaee-23-Conjectures}. In a sequel to this article, \cite{Mackenzie-Rezaee2}, we prove a generalized version of this conjecture for certain non-tetrahedral $l$, via proving the conjectural necessary condition. 
\end{abstract}

{\hypersetup{linkcolor=black}
\tableofcontents}

\section{Introduction} 

The geometry of Hilbert schemes of points $\mathrm{Hilb}^l(\BA^N)$ has been studied for decades.  Perhaps, the longest-standing open problem regarding the singularities  of $\mathrm{Hilb}^l(\BA^N)$  is a conjecture by Brian\c{c}on and Iarrobino in the 1970s, which predicted the maximum singularity when $l$ is a (hyper)tetrahedral number. In this article, we give a short proof of this conjecture for $N=3$.

\subsection{History and relevant work}

Since its invention by Grothendieck in the 1960s, \cite{Grothendieck}, the Hilbert scheme has been ubiquitous and studied in various fields in mathematics, including algebraic geometry (in the areas such as enumerative geometry and birational geometry), commutative algebra, combinatorics, representation theory, gauge theory, (e.g., see \cite{Grojnowski}), and knot theory, (e.g., see \cite{GORSKY2015403}). Because of its highly singular nature, studying the global geometry of the Hilbert scheme is highly complicated. To measure how significant the singularities are, we need to measure the dimension of the tangent space at a given point of the Hilbert scheme. The maximal singularity occurs at the points with the maximal dimension of the tangent space.

One of the major results on the global geometry of the Hilbert scheme appeared in the Ph.D. thesis of Hartshorne, \cite{HartshorneConnectedness} in the 1960s,  proving the connectedness of the Hilbert scheme. In the late 1960s, Fogarty showed the smoothness of the Hilbert scheme for $N=2$, \cite{Fogarty}. In the early 1970s, Iarrobino found the first example of a reducible Hilbert scheme of points in three dimensions for the case of $l=78$, \cite{Iarrobino72} (other works on the (ir)reducibility in three dimensions include \cite{Hilb8} in the late 2000s, and \cite{Jardim, Hilb_11} in the 2010s). In the late 1970s, Brian\c{c}on and Iarrobino predicted that ideals with maximal number of generators give the maximal singularity, while Sturmfels disproved it, \cite{Sturmfels}. In the early 2000s, Haiman used the Hilbert scheme to resolve a combinatorial conjecture, \cite{Haiman2} (also see \cite{MillerSturmfels2005}). In the 2000s, Vakil proved Murphy's law for Hilbert schemes of positive dimensional subspaces \cite{Vakil06},  and more recently, Jelisiejew proved it for the Hilbert scheme of points \cite{Jelisiejew20}.  In the 2000s, Sturmfels suggested that the maximal singularity of $\mathrm{Hilb}^l(\BA^3)$ is attained at an initial
monomial ideal of the generic configuration of $l$ points, which was recently disproved in \cite{Ramkumar-Sammartano}. For connections to enumerative geometry, we refer to \cite{BBS, BFHilb,JKSCounterBehrend, MNOP1, Nesterov25, PT, RicolfiSign} and for the parity conjecture, we refer to \cite{PandharipandeSlides,Ramkumar-Sammartano1,GGGL23}. For a recent work on the irrational components of the Hilbert scheme of points, we refer to \cite{Farkas-Pandharipande-Sammartano24}, and for rational singularities, we refer to \cite{Ramkumar-Sammartano24}. For further works on the case of 3-folds, we refer to \cite{KatzS, HuX}. Among the more recent works, in \cite{Jelisiejew-Ramkumar-Sammartano24}, the authors consider the opposite problem of identifying the least singular points, i.e., the smooth point in three dimensions. For the birational geometry of $\Hilb^l(\BP^2)$, we refer to \cite{ABCH, Huizenga-2016, Huizenga-thesis, LZ, LZ2}. In upcoming articles by Gross and the second-named author (\cite{Gross-Rezaee1, Gross-Rezaee2}), the authors use the machinery of scattering diagrams to describe the birational geometry of the Hilbert scheme $\Hilb^l(\mathbb{P}^2)$. Also, in \cite{Conj-ARZ}, the authors proposed a generalized version of the Brian\c{c}on-Iarrobino Conjecture.

The latest progress regarding the Brian\c{c}on-Iarrobino Conjecture was made by Ramkumar and Sammartano in \cite{Ramkumar-Sammartano} (which in turn was inspired by the work of Haiman, \cite{HAIMAN1998201}, on $N=2$), where the authors had proven the predicted upper bound for the tangent space up to a factor of $4/3$. Our approach is based on the decomposition of the tangent space to the Hilbert scheme, introduced by Ramkumar and Sammartano.

The novelty in our work is that we introduce an error term, which makes the existing upper bound stronger; hence, this immediately resolves the conjecture of Brian\c{c}on and Iarrobino in three dimensions. Our approach will also shed light on proving the conjectural necessary condition, \cite[Conjecture B]{Rezaee-23-Conjectures}, for certain non-tetrahedral cases in three dimensions (the work on this will appear in a sequel, \cite{Mackenzie-Rezaee2}). {This may also hint at resolving the Brian\c{c}on-Iarrobino in higher dimensions, which will be considered in our future work.}

\begin{remark}
    Throughout the paper, we can replace $\BC$ by any characteristic $0$ field. Also, we only consider the minimal generators of ideals. 
\end{remark}

\subsection{Basic definitions} We recall the basic definitions.

To a $0$-dimensional ideal\footnote{An ideal $I$ such that the quotient $\BC[x,y,z]/I$ is Artinian (zero-dimensional).} $I$ in $R=\BC[x,y,z]$, we associate the corresponding point $[I]$ in the Hilbert scheme $\CH=\Hilb^l(\BA^3)$. We denote the dimension of the tangent space to the Hilbert scheme $\CH$ at $[I]$  by $T(I)$.

For a $0$-dimensional Borel-fixed ideal $I=(x^{m_1},y^{m_2},z^{m_3},\text{mixed generators})$, the exponents $m_i$ are called \emph{pure exponents}.

\begin{definition}[Maximal singularity]
    By a point $[I]$ in $\CH=\Hilb^l(\BA^3)$ \emph{having the maximal singularity} we mean that $T(I)\geq T(J)$ for any $[J]\in\CH$. 
    
    By convention, we use \emph{maximum singularity} if there is only one option.

\end{definition}

The following definition is crucial.

\begin{definition}[Borel-fixedness]\label{Def:Borel-fixedness} An ideal $I$ is called a Borel-fixed ideal,
if for any generator $g$ of $I$, we have 
\begin{itemize}
    \item if $z|g$, then the quotients $\frac{xg}{z}, \frac{yg}{z} \in I$, 
    \item if $y|g$, then the quotient  $\frac{xg}{y} \in I$. 
    \end{itemize}
\end{definition}

\begin{remark}\label{rem:Borel-fixed only}
    Note that for any $l\in \BZ^{\geq 1}$, there exists a Borel-fixed ideal $I$ (Definition \ref{Def:Borel-fixedness}) of colength $l$ which corresponds to a point of $\Hilb^l(\BA^N)$  with the maximal singularity. So, in this article, we only focus on Borel-fixed ideals.
    \end{remark}

\subsection{Main results} We prove the following results in Section \ref{Section: proofs}.

\begin{thm}[Upper bound]\label{thm:upperbound}
    Let 
    \begin{align*}
        I=(x^{m_1},y^{m_2},z^{m_3},\text{mixed generators})
    \end{align*}
    be a 0-dimensional Borel-fixed ideal in $\BC[x,y,z]$ of colength $l=\binom{k+2}{3}+\Delta$, where $k\in\BZ^{\geq1}$ and $0\leq \Delta \leq \binom{k+2}{2}-{1}$. Then 
    \begin{align*}
        T(I)
        \leq 
    (2m_1+1)l-2\binom{m_1+2}{4}. 
    \end{align*}
\end{thm}

\begin{definition*}[Function $\psi$]
For fixed $l={k+2 \choose 3}+\Delta$, and any Borel-fixed ideal $$I=(x^{m_1},y^{m_2},z^{m_3},\text{mixed generators})$$ in $\BC[x,y,z]$ of colength $l$,  
we define the function
\begin{align}
    \psi(m_1):=(2m_1+1)l-2{m_1+2 \choose 4}, 
\end{align}
for $\binom{m_1+2}{3}\leq l$.
\end{definition*}

Given this definition, we will prove the following crucial result.

\begin{thm}[Monotonicity of $\psi$]\label{thm:IncreasingU}
    For fixed $l$, the function $\psi$ is strictly increasing.
\end{thm}

For a tetrahedral $l$, the upper and lower bounds will coincide, so Theorem \ref{thm:IncreasingU} immediately implies the main conjecture:

\begin{cor}[3D Brian\c{c}on-Iarrobino Conjecture]\label{Cor:BIConj} For $N=3$, the Brian\c{c}on-Iarrobino Conjecture, \cite[Section III]{Bri-Iar} holds: The maximum singularity of  $\text{\:}\mathrm{Hilb}^{{k+2\choose 3}}(\BA^3)$ occurs at $[\mathfrak{m}^k]$.
\end{cor}
\setcounter{conj}{2}

Theorem \ref{thm:IncreasingU} also implies the following conjectural necessary condition for a tetrahedral $l$:
\subsection*{Conjecture B \cite[Conjecture B for $N=3$]{Rezaee-23-Conjectures}}\label{necessaryCondition} \textit{Let $I=(x^{m_1},y^{m_2}, z^{m_3},\text{mixed generators})$ be a 0-dimensional Borel-fixed ideal of colength $l$ in $\BC[x,y,z]$, where ${2+k \choose 3}\leq l<{3+k\choose 3}$. If $[I]$ is a maximal singularity of $\text{\:}\mathrm{Hilb}^{l}(\BA^N)$, then $m_1=k$.}

\begin{cor}[Necessary condition]\label{Cor:ConjB}
For $l={k+2 \choose 3}$, the conjectural necessary condition for maximum singularity, \cite[Conjecture B]{Rezaee-23-Conjectures}, holds: the smallest pure exponent is equal to $k$.
\end{cor}

\subsection{Outline of the proof of Brian\c{c}on-Iarrobino Conjecture} The idea is surprisingly simple: Fix a colength $l$. We give an upper bound for the dimension of the tangent space to $\mathrm{Hilb}^l(\BA^3)$ at any given point $[I]$ of the Hilbert scheme by $(2m_1+1)l-2{m_1+2 \choose 4}$, where $m_1$ is the minimal pure exponent of $I$. When $l={k+2 \choose 3}$ and $I=\Fm^k$, this upper bound coincides with $T(I)$, which immediately resolves the desired conjecture.

\subsection*{Acknowledgements} This is an outcome of the Cambridge Summer Research In Mathematics (SRIM) programme 2025.  We are very grateful to Mark Gross and the SRIM for supporting the project. We especially thank Tony Iarrobino, Dominic Joyce and Alessio Sammartano for careful reading and several helpful comments on the first draft. We also thank Alexia Ascott for helpful discussions and Zac Owen and Marek Szuba for their help with various IT matters. OM was supported by the ERC Advanced Grant MSAG, and FR was supported by the UKRI grant EP/X032779/1. We acknowledge the use of Macaulay2 \cite{M2} and Python.

\subsection*{Notation} Let $I$ be a $0$-dimensional ideal in $R=\BC[x,y,z]$.
\begin{center}
    
         \begin{tabular}{ r l }

 $\mathrm{hom}(C,D)$:& The dimension of the space $\mathrm{Hom}_{R}(C,D)$ over $\BC$.\\

          $l(I)$:& The colength of the ideal $I$ which is defined by $\mathrm{hom}(R,R/{I})$.\\

              $[I]$:& The point of the Hilbert scheme corresponding to the defining ideal $I$.\\
          
           $T(I)$:& The dimension of the  tangent space, $\mathrm{Hom}_R(I,R/I)$, to the Hilbert scheme at\\&  $[I]$, which is defined by $\mathrm{hom}(I,R/{I})$.\\

         $\mathfrak{m}$:&The maximal ideal of  $R$ given by  $(x,y,z)$.\\

          $\CP(C)$:& Power set of a set $C$.\\
          $B^{J,J'}_n\setminus A^{J,J'}_n$:& The set of \emph{ghost vectors} (see Definition \ref{Def:ghostVector} and the Definition \ref{Def:AJ}).\\
          $\tilde{I}$:& The set $\{\gamma\in\mathbb{N}^3:\mathbf{x}^\gamma\in I\}$  (inspired by the notation in \cite{Ramkumar-Sammartano}).
          \\
$\tilde{I}_\alpha$:&  The set $(\tilde{I}+\alpha)\setminus\tilde{I}$, for $\alpha$ a vector in $\BZ^3$.

                      \end{tabular}
     \end{center}

\subsection*{Conventions} We have the following conventions.

\begin{itemize}
    \item We assume the convention $\BN=\{0,1,2,\ldots\}$.

    \item  For integers $m<n$, we use the convention that ${m \choose n}=0$.

    \item The ideals we consider in this article are always $0$-dimensional, by this we mean that the quotient $R/I$ is Artinian (zero-dimensional).
\end{itemize}

\section{Background}
In this section, we provide some background material, definitions and statements which will be used in the following section. Let $I$ be a $0$-dimensional monomial ideal in $R=\BC[x,y,z]$. We recall that by Remark \ref{rem:Borel-fixed only}, we only need to focus on Borel-fixed ideals in order to prove our results on maximum singularity.

\begin{definition}[Ghost vector]\label{Def:ghostVector} If $g$ is a generator of a Borel-fixed $I$, and $q$ is a lattice point in the $\mathrm{xy^{-}z}$-octant, with the $y$-coordinate equal to $-1$, under some additional conditions (as in Definition \ref{Def:AJ}), we call the vector from $g$ to $q$, a \emph{ghost vector}. These vectors are equivalent to the elements in $B^{J,J'}_n\setminus A^{J,J'}_n$. See Figure \ref{Fig: Ghost-Zero-Vector} for $I=(x^2,y^2,z^3,xz,yz^2,xy)$ and a ghost vector $(1,-2,-1)$.
    
\end{definition}

\begin{definition}[Zero vector]\label{def:ZeroVector} Let $I$ be a Borel-fixed ideal in $R=\BC[x,y,z]$ of colength $l$, with the minimal  number of generators. Then the vectors from any generator of $I$ to $R/I$ which do not belong to $\Hom(I,R/I)$ are called \emph{zero vectors}. See Figure \ref{Fig: Ghost-Zero-Vector} for $ I=(x^2,y^2,z^3,xz,yz^2,xy)$, and a zero vector $(0,-1,-2)$
    
\end{definition}

 \begin{figure}[h]
 \subcaptionbox*{}[.5
 \linewidth]{%
    \includegraphics[width=\linewidth]{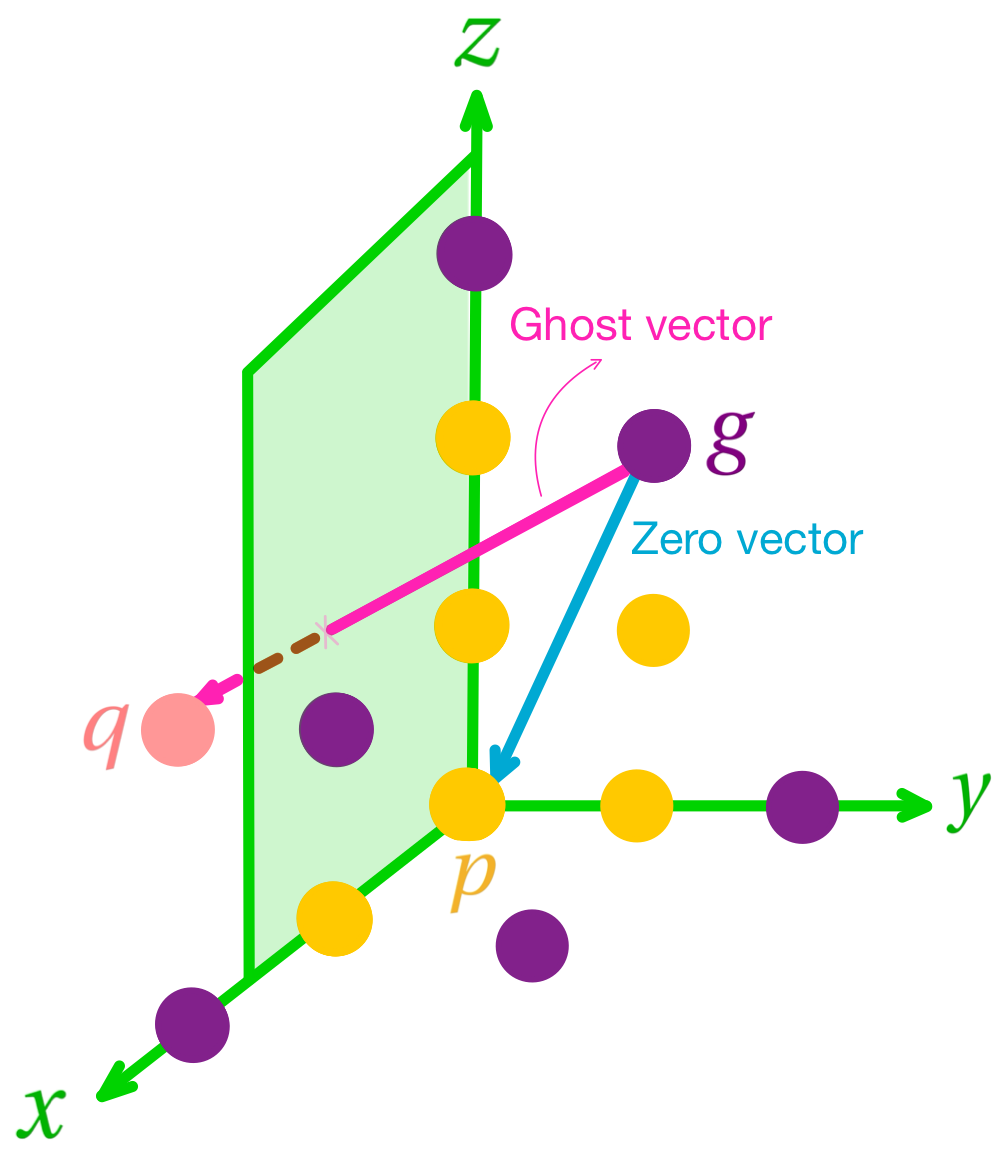}%
  }
  \vspace{-0.6cm}
  \caption{An example of a zero vector $(0,-1,-2)$ and a ghost vector $(1,-2,-1)$.}
  \label{Fig: Ghost-Zero-Vector}
\end{figure}

\subsection{Filtration}Let $I$ be a $0$-dimensional monomial ideal in $\BC[x,y,z]$. As introduced in \cite{Ramkumar-Sammartano}, we use the decomposition of $I$ into monomial ideals in $\BC[y,z]$. 

\begin{definition}\label{Def: decomposition}
    We decompose $I$ an ideal in $\BC[x,y,z]$  into ideals $I_i$ in $\mathbb{C}[y,z]$ as $I=\bigoplus_ix^iI_i$. 
\end{definition}
\begin{example}\label{Ex:first}
    See Figure \ref{Fig: Filtration} for the ideal $I=(x^2,y^2,z^3,xz,yz^2,xy)$, with $I_0=(y^2,z^3,yz^2)$, $I_1=(y,z)$ and $I_2=(1)$.
    
    Also, see Figure \ref{Fig: Filtration'} for the ideal $I=(x^3,y^4,z^6,y^3z,y^2z^3,yz^5,xy^2,xz^4,xyz^2,x^2y,x^2z)$, with $I_0=(y^4,z^6,y^3z,y^2z^3,yz^5)$, $I_1=(y^2,z^4,yz^2)$, $I_2=(y,z)$ and $I_3=(1)$.
  \begin{figure}[h]
 \subcaptionbox*{(i) $I$}[.31\linewidth]{
    \includegraphics[width=\linewidth]{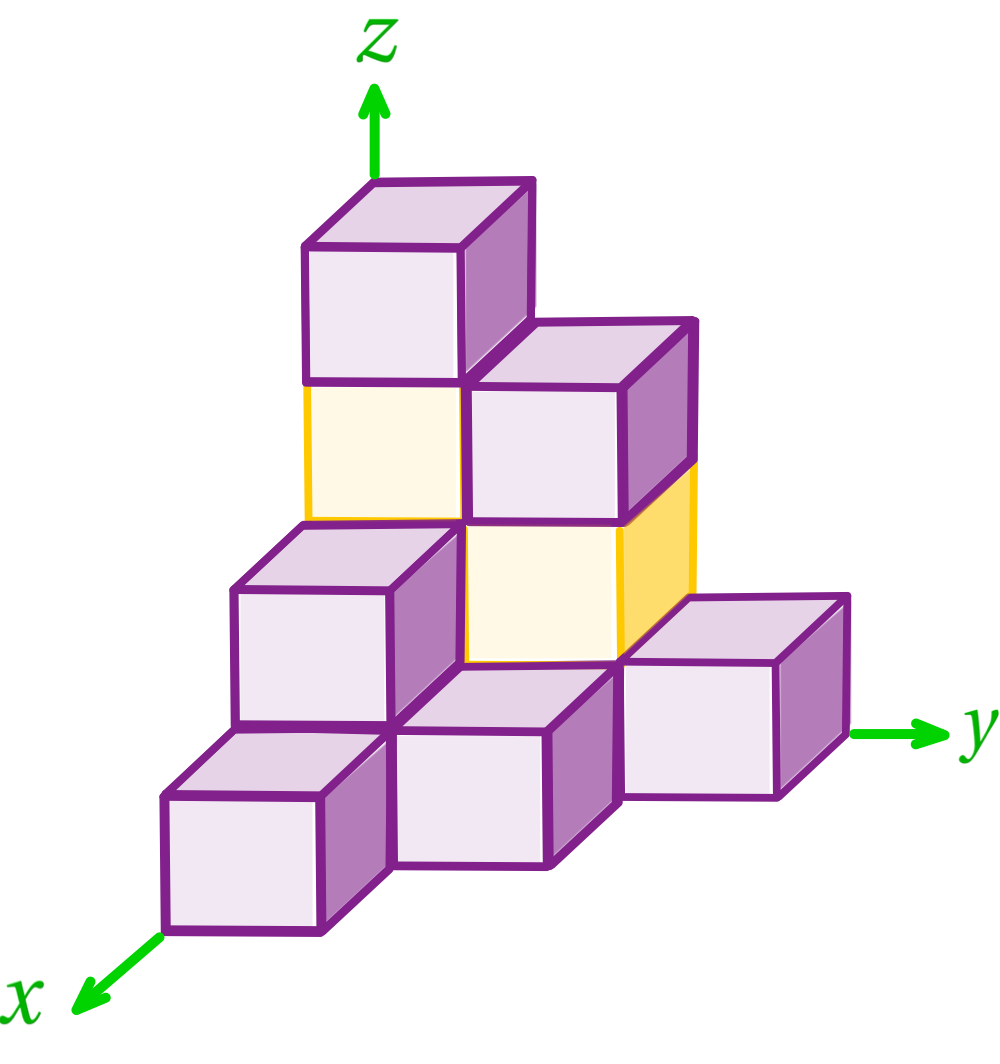}
  }
  \hskip15ex
  \subcaptionbox*{(ii) $I=\bigoplus_i x^iI_i$}[.28\linewidth]{
    \includegraphics[width=\linewidth]{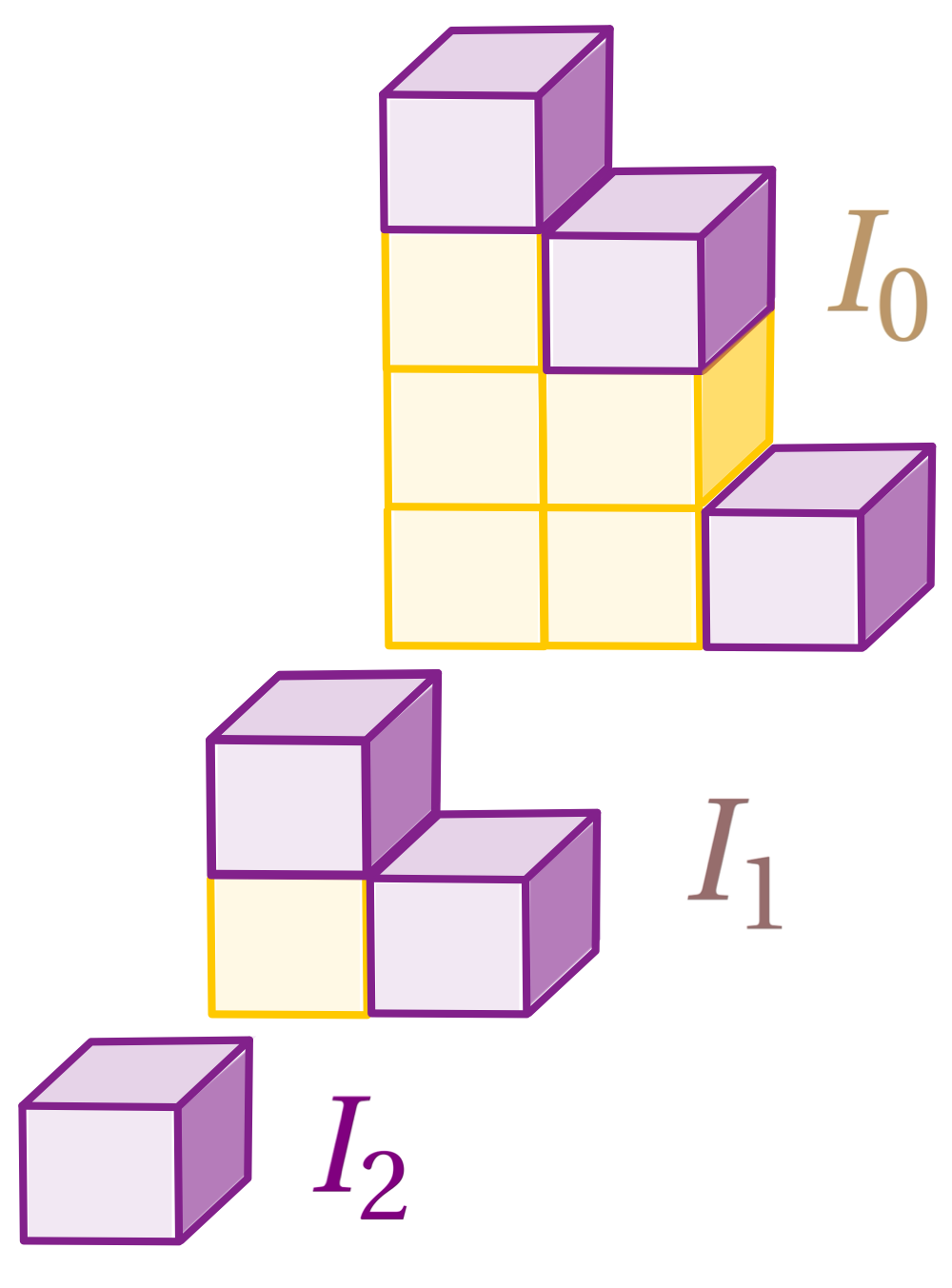}
  }
  \caption{Decomposition of the ideal $I=(x^2,y^2,z^3,xz,yz^2,xy)$.}
  \label{Fig: Filtration}
\end{figure}
 \begin{figure}[h]
  
 \subcaptionbox*{(i) $I$}[.42\linewidth]{
    \includegraphics[width=\linewidth]{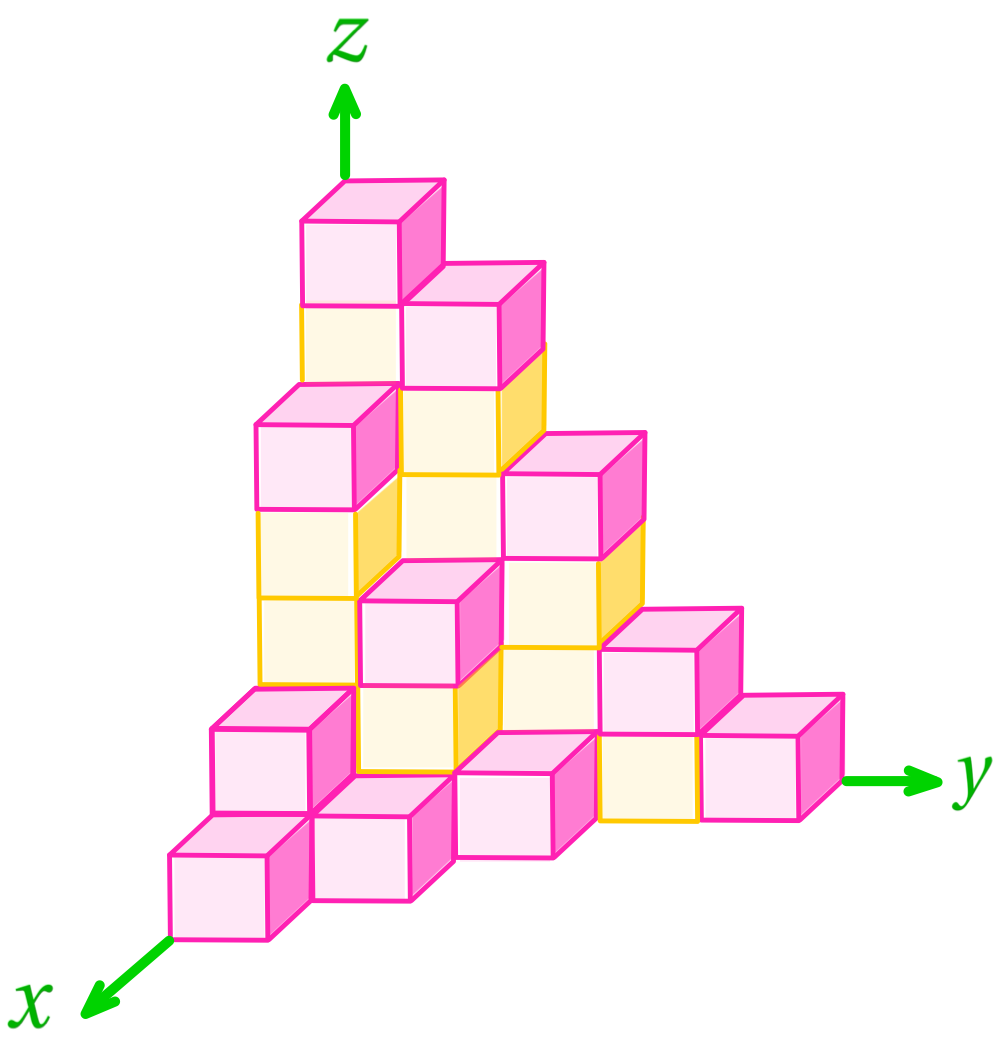}
  }
  \hskip1ex
  \subcaptionbox*{(ii) $I=\oplus_i x^iI_i$}[.49\linewidth]{
    \includegraphics[width=\linewidth]{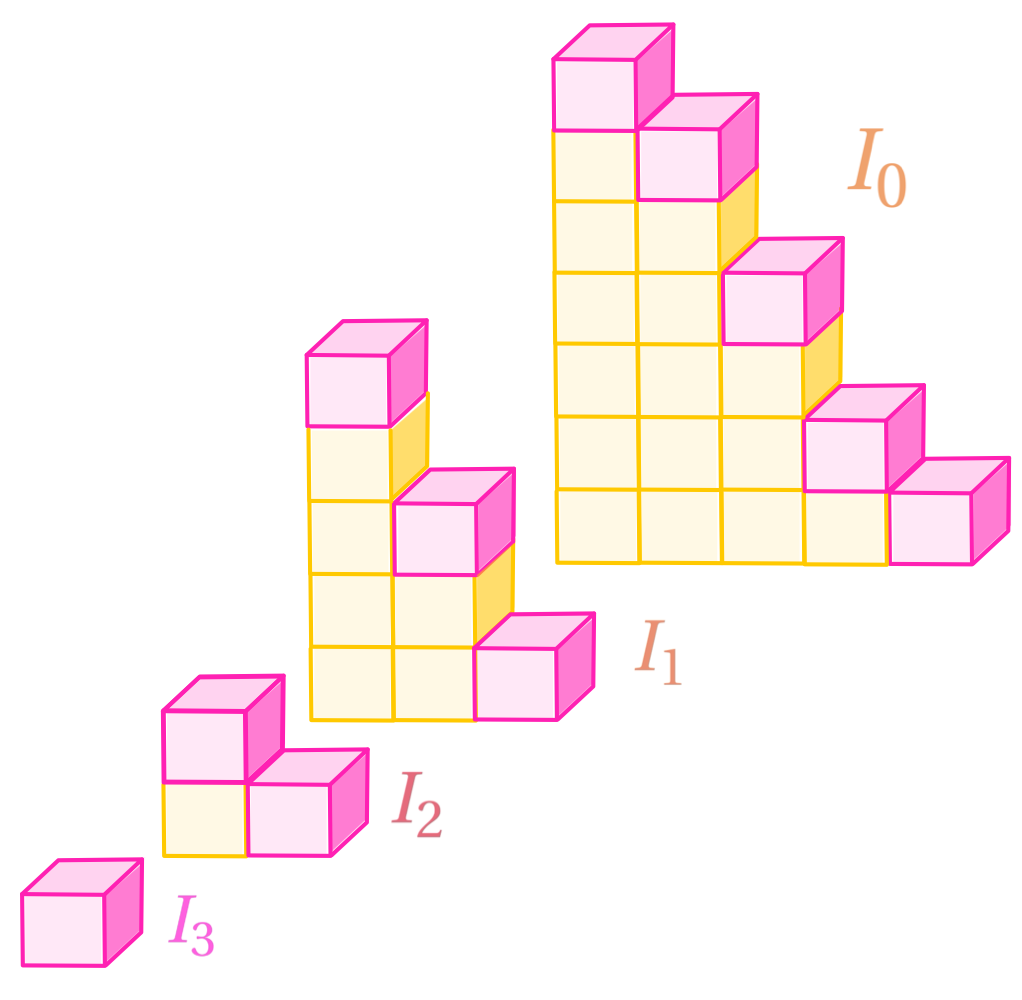}
  }
  \caption{Decomposition of the ideal $I=(x^3,y^4,z^6,y^3z,y^2z^3,yz^5,xy^2,xz^4,$\\$xyz^2,x^2y,x^2z)$}
  \label{Fig: Filtration'}
\end{figure}
\end{example}

Definition \ref{Def: decomposition} provides an incentive for the following lemma.

\begin{lemma}\label{Lem:inequalityL1L2}
    Let $J$ and $J'$ be two $0$-dimensional monomial ideals in $S=\BC[y,z]$ of colengths $l$ and $l'$, respectively. Then, we have
    \begin{align}\label{eq: Ineq}
       \mathrm{hom}(J,S/J')\leq l+l'.
    \end{align}
\end{lemma}

\begin{remark}Note that we can still apply the method of calculating $\mathrm{hom}(J,S/J)$ used in \cite{Ramkumar-Sammartano} to calculate $\mathrm{hom}(J,S/J')$. Thus, the dimension of $\mathrm{hom}(J,S/J')$ in Lemma \ref{Lem:inequalityL1L2} will be calculated by summing the number of bounded connected components of $(\tilde{J}+\alpha)\setminus \tilde{J'}$ across all $\alpha=(\alpha_y,\alpha_z)\in\BZ^2$.
\end{remark}
Before proving Lemma \ref{Lem:inequalityL1L2}, we have the following definition.

\begin{definition}\label{Def:AJ}
The space $H^{J,J'}:=\mathrm{Hom}(J,S/J')$ is graded by $\mathbb{Z}^2$ according to the value of $\alpha$ used. We divide $H^{J,J'}$ into two subspaces, $H^{J,J'}=H^{J,J'}_p\oplus H^{J,J'}_n$ as
\[H^{J,J'}_p:=\bigoplus_{\alpha\in\mathbb{Z}\times\mathbb{N}}H_\alpha,\]
\[H^{J,J'}_n:=\bigoplus_{\alpha\in\mathbb{Z}\times\mathbb{Z}^-}H_\alpha.\]
We will see that that $\dim(H^{J,J'}_p)=l'$ and $\dim(H^{J,J'}_n)\leq l$. 

Let
\begin{align*}   
A^{J,J'}:=\{(U,\alpha)\in\mathcal{P}(\mathbb{N}^2)\times\mathbb{Z}^2:U\text{ is a connected component of }(\tilde{J}+\alpha)\setminus \tilde{J'}\}.\end{align*}

Lowering the grading $n,p$ to this set,  we get $A^{J,J'}_p$ and $A^{J,J'}_n$ corresponding to bases of $H^{J,J'}_p$ and $H^{J,J'}_n$ as follows. 
\[A^{J,J'}_p:=\{(U,\alpha)\in\mathcal{P}(\mathbb{N}^2)\times\mathbb{Z}^2:U\text{ is a connected component of }(\tilde{J}+\alpha)\setminus \tilde{J'},\alpha_z\geq 0\},\]
\[A^{J,J'}_n:=\{(U,\alpha)\in\mathcal{P}(\mathbb{N}^2)\times\mathbb{Z}^2:U\text{ is a connected component of }(\tilde{J}+\alpha)\setminus \tilde{J'},\alpha_z<0\}.\]

Next, we define the set $B^{J,J'}_n$ as an extension of $A^{J,J'}_n$ by
\begin{align*}    
B^{J,J'}_n=:\{(U,\alpha)\in\mathcal{P}(\mathbb{Z}\times\mathbb{N})\times\mathbb{Z}^2:U\text{ is a connected component of } ( \tilde{J}+\alpha)\setminus\tilde{J'},\alpha_z<0\}.\end{align*}
Note that $B^{J,J'}_n$ corresponds to the not necessarily bounded components of $(\tilde{J}+\alpha)\setminus\tilde{J'}$ with $\alpha_z<0$, which are contained entirely in the upper-half-plane.

\end{definition}

\begin{proof}[Proof of Lemma \ref{Lem:inequalityL1L2}]
 
First, we biject $A^{J,J'}_p$ with $\mathbb{N}^2\setminus\tilde{J'}$, which in turn corresponds to the natural basis of $S/J'$; hence $A^{J,J'}_p$ will be of
size $l'$. To construct such a bijection,  for $(U,\alpha)\in A^{J,J'}_p$, we choose  $\gamma={(\gamma_y,\gamma_z})$ in $\tilde{J}$ with maximal $z$-coordinate and $\gamma+\alpha\in U$. If there is more than one option, we take the one with minimal $y$-coordinate. Similarly, we have that $\gamma_y+\alpha_y$ is the smallest $y$-coordinate of any element of $U$.

Define
\begin{align*}
    f:A^{J,J'}_p&\rightarrow \mathbb{N}^2\setminus\tilde{J'}\\
    (U,\alpha)&\mapsto(\gamma_y+\alpha_y,\alpha_z)
\end{align*}

\emph{Claim 1.} $f$ is a well-defined bijection.
\begin{proof}[Proof of Claim 1] We prove this in three steps:

\begin{enumerate}
    \item $f$ is well-defined: $\gamma_y+\alpha_y\in\BN$, and $\alpha_z\in \BN$ are clear. If $f(U,\alpha)\in\tilde{J'}$, then since $\gamma_z\geq0$, we must have $\alpha+\gamma\in\tilde{J'}$ and $\alpha+\gamma\in U$; contradiction.

     \item $f$ is a surjection: Given an element $\beta=(\beta_y,\beta_z)$ in $\mathbb{N}^2\setminus\tilde{J'}$, we have $\beta_z\geq 0$. Hence, we can define $\alpha_z:=\beta_z$. Now, we define $\alpha_y$ in the following way. Take $\alpha_y$ to be maximal with the property that there are elements with $y$-coordinate equal to $\beta_y$ in $(\tilde{J}+\alpha)\setminus\tilde{J'}$. Now, take $U$ to be the unique component of $(\tilde{J}+\alpha)\setminus\tilde{J'}$ which contains the elements with  $y$-coordinate equal to $\beta_y$. If $U$ contains elements with $y$-coordinate less than $\beta_y$, we arrive at a contradiction because we would have two points of coordinates $(\beta_y,s),(\beta_y-1,s)\in U$ so we can increase $\alpha_y$ by at least 1. Note that since $\alpha_z\geq 0$, $U$ is restricted to the upper-half-plane. On the other hand, since  $\beta_y\geq 0$, $U$ is also restricted to the right-half-plane. Hence, $U$ is bounded, and we have $f(U,\alpha)=\beta$.

     \item $f$ is an injection: Suppose $\beta=(\gamma_y+\alpha_y,\alpha_z)=f(U,\alpha)=f(\bar{U},\bar{\alpha})=(\bar{\gamma}_y+\bar{\alpha}_y,\bar{\alpha}_z)$, for some $\bar{\alpha}=(\bar{\alpha}_y,\bar{\alpha}_z)$, $\bar{\gamma}=(\bar{\gamma}_y,\bar{\gamma}_z)$ and $\bar{U}$, defined similarly as $\alpha$, $\gamma$ and $U$. It is clear that $\alpha_z=\bar{\alpha}_z$. For the sake of contradiction, suppose $\alpha_y<\bar{\alpha}_y$; then $\gamma_y>\bar{\gamma}_y$. Now $\bar{\gamma}+\bar{\alpha}=\bar{\gamma}+\alpha+r\mathbf{e}_y\notin\tilde{J}'$ where $r=\bar{\alpha}_y-\alpha_y>0$. We deduce that $\bar{\gamma}+\alpha\in U$, which contradicts the definition of $\gamma$. Hence, $\alpha=\bar{\alpha}$. Finally, $U$ and $\bar{U}$ are the unique connected component of $(\tilde{J}+\alpha)\setminus\tilde{J}'$, which contain elements with $y$-coordinate equal to $\beta_y$. Thus $U=\bar{U}$, and we have an injection as required.
\end{enumerate}
\end{proof}

We can similarly biject $B^{J,J'}_n$ with $\mathbb{N}^2\setminus\tilde{J}$. For  this purpose, associated to $(U,\alpha)\in B^{J,J'}_n$, we choose $\gamma=(\gamma_y,\gamma_z)$ in $\tilde{J}$ such that $\gamma+\alpha\in U$, and $\gamma$ has maximal $y$-coordinate. If there are multiple options, we take the one with minimal $z$-coordinate. 

Define
\begin{align*}
    g:B^{J,J'}_n&\rightarrow \mathbb{N}^2\setminus\tilde{J}\\
    (U,\alpha)&\mapsto (\gamma_y,-\alpha_z-1)
\end{align*}

\emph{Claim 2.} $g$ is a well-defined bijection.

\begin{proof}[Proof of Claim 2] Again, we proceed with the proof in three steps:

\begin{enumerate}
    \item $g$ is well-defined: It is clear that $\gamma_y \in \BN$ and $-\alpha_z-1\in\BN$. Since $U$ is in the upper-half-plane, we have $\gamma_z+\alpha_z+1> 0$. Therefore, $\gamma_z>-\alpha_z-1$, and so $f(U,\alpha)\notin\tilde{J}$.
    \item $g$ is a surjection: Given an element $\beta$ in $\mathbb{N}^2\setminus\tilde{J}$, we can find $\gamma=(\gamma_y,\gamma_z)\in J$ such that $\gamma_y=\beta_y$ and $\gamma_z$ is minimal. We set $\alpha_z:=-\beta_z-1$. Then, we can find the unique $\lambda$ such that $(\lambda,\gamma_z+\alpha_z)\notin\tilde{J'}$ but $(\lambda+1,\gamma_z+\alpha_z)\in\tilde{J'}$. Then, we can uniquely define $\alpha_y:=\lambda-\gamma_y$ and $\alpha=(\alpha_y,\alpha_z)$. Choosing $U$ to be the connected component of $(\tilde{J}+\alpha)\setminus\tilde{J}'$ containing $\alpha+\gamma$, then as before, $\gamma_z+\alpha_z$ is the minimal $z$-coordinate of an element of $U$. So, this gives $f(U,\alpha)=\beta$.
    \item $g$ is an injection: This part is similar to the injection of $f$ as above.
\end{enumerate}
    
\end{proof}

Combining Claim 1 and Claim 2, we have
\begin{align*}
    \mathrm{hom}(J,S/J')= l+l'-\mathrm{Card}( B^{J,J'}_n\setminus A^{J,J'}_n),
    \end{align*}
which implies the desired inequality.
\end{proof}

\begin{remark}
    The bound \eqref{eq: Ineq} is related to  \cite[Proposition 4.1]{Ramkumar-Sammartano}. We prove \eqref{eq: Ineq}  via a different method, which is purely combinatorical. This approach allows us to improve the bounds by subtracting further terms, which we can evaluate easily in order to eventually prove the original conjecture.
\end{remark}

\begin{corollary} \label{Cor:An'An}
    For $J$ and $J'$ as above, we have
    \begin{align}
        \mathrm{hom}(J,S/J')=l+l'-\mathrm{Card}( B^{J,J'}_n\setminus A^{J,J'}_n)
    \end{align}
\end{corollary}
\begin{proof}
    This is immediate from the proof of Lemma \ref{Lem:inequalityL1L2}.
\end{proof}

Moreover, Lemma \ref{Lem:inequalityL1L2} immediately implies the following well-known result about the Hilbert scheme of $\BA^2$.
\begin{corollary}
     $\Hilb^l(\mathbb{A}^2)$ is smooth.
\end{corollary}
\begin{proof}
    Taking $I=J=J'$ we get the inequality $T(I)\leq 2l$ which suffices (since $2l \leq T(I)$) to deduce the smoothness of $\Hilb^l(\mathbb{A}^2)$.
\end{proof}
\begin{definition}
    Let $J$ be a $0$-dimensional monomial ideal in $\BC[y,z]$ we define its \emph{height} to be $h$ such that $z^h$ is a minimal generator.
\end{definition}
\begin{definition}
    Let $I_0,I_1$ be two $0$-dimensional monomial ideals in $\BC[y,z]$. We say that $I_0$ is \emph{taller than} $I_1$ (or equivalently, $I_1$ is \emph{shorter than} $I_0$) if the height of $I_0$ is greater than that of $I_1$.
\end{definition}

\begin{definition} Let $I_0,I_1$ be two $0$-dimensional  monomial ideals in $\BC[y,z]$ such that $I_1$ has height $h$. We define $t(I_0,I_1)$ to be the size of the set $\{\beta\in\mathbb{N}^2\setminus\tilde{I_0}:\beta_z\geq h\}$.

    Note that if $I_1$ is taller than $I_0$, then $t(I_0,I_1)=0$.
\end{definition}

\begin{lemma}\label{Cor:tLessGhost}
 For $J$ and $J'$ as above, we have $t(J,J') \leq \mathrm{Card}(B^{J,J'}_n\setminus A^{J,J'}_n)$.
\end{lemma}

\begin{proof}

   Let $h$ be the height of $J'$. Each point $\beta=(\beta_y,\beta_z)$ counted by $t(J,J')$ can be mapped to $(-1,h-1)$ and these correspond to distinct elements of $B^{J,J'}_n\setminus A^{J,J'}_n$, although not necessarily all of them. Since the value of $\alpha=(\alpha_y,\alpha_z)$ used is $(-1,h-1)-\beta=(-1-\beta_y,h-1-\beta_z)$, we have $\alpha_z=h-1-\beta_z<0$, since by definition, $\beta_z\geq h$. Hence, it suffices to check that there is a connected component $U$ of $(\tilde{J}+\alpha)\setminus\tilde{J'}$, which is unbounded but only inside the upper-half-plane. First, we note that $\alpha_y=-1-\beta_y<0$ so there is certainly an unbounded component $U$ in the upper-left quadrant. However, by the definition of $\beta$ and $h$, we have $\beta\notin\tilde{J}$ and $(0,h)\in\tilde{J'}$, respectively. The former implies that there is no $(\Fy,\Fz)$ in $(\tilde{J}+\alpha)$ such that $\Fy\leq-1,\Fz\leq h-1$, as $\beta+\alpha=(-1,h-1)$. The latter implies that for any $(\Fy,\Fz)$ such that $\Fy\geq 0,\Fz\geq h$, we have $(\Fy,\Fz)\in\tilde{J'}$. Combining these implies that $U\subset (-\infty,-1]\times[h,\infty)$; thus, this corresponds to an element of $B^{J,J'}_n\setminus A^{J,J'}_n$. Note that for different values of $\beta$ we get different values of $\alpha$, so we get $t(J,J')$ distinct elements of $B^{J,J'}_n$.
\end{proof}

The following statement is immediate from Corollary \ref{Cor:An'An} and Lemma \ref{Cor:tLessGhost}. 

\begin{corollary}\label{cor:ineq}
    For $J$ and $J'$ as above, we have $\mathrm{hom}(J,S/J')\leq l+l'-t(J,J')$.
\end{corollary}
\begin{example}
  Let $I_0=(y^2,z^3,yz^2)$ and $I_1=(y,z)$ (as in Example \ref{Ex:first}). We have $h=1$.   The ghost vectors are $(-1,-2)$, $(-2,-1)$ and $(-1,-1)$, and these are all counted by $t(I_0,I_1)$. Hence, we have $\mathrm{Card}(B^{I_0,I_1}_n\setminus A^{I_0,I_1}_n)=3$ and $t(I_0,I_1)=3$.  See Figure \ref{Fig: t(I0,I1)}. 
  
  As another example, let $I_0=(y^4,z^6,y^3z,y^2z^3,yz^5)$ and $I_1=(y^2,z^4,yz^2)$ (as in Example \ref{Ex:first}). We have $h=4$.  Explicitly, the ghost vectors are  $(-1,-1)$, $(-1,-2)$, $(-2,-1)$ and $(-1,-3)$. Among these,  the first three vectors are counted by $t(I_0,I_1)$. Hence, we have $\mathrm{Card}(B^{I_0,I_1}_n\setminus A^{I_0,I_1}_n)=4$ and $t(I_0,I_1)=3$.
  This example shows that the inequality in Lemma \ref{Cor:tLessGhost} can be strict.  See Figure \ref{Fig: t(I0,I1)'}.

     \begin{figure}[h]
\subcaptionbox*{}[.85
 \linewidth]{%
\includegraphics[width=\linewidth]{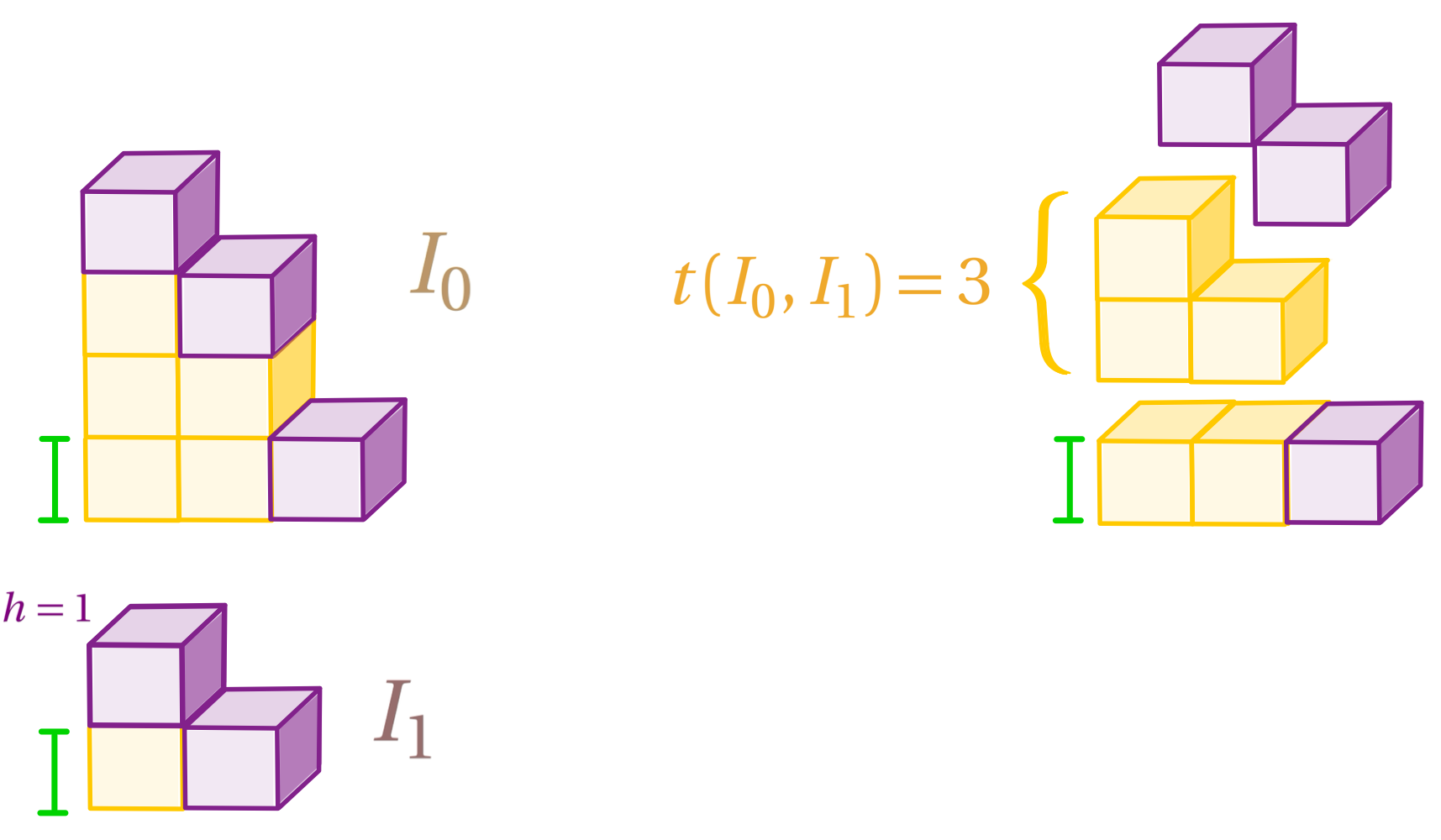}%
  }
  \vspace{-0.6cm}
  \caption{For $I_0=(y^2,z^3,yz^2)$ and $I_1=(y,z)$, we have $h=1$ and $t(I_0,I_1)=\mathrm{Card(B^{I_0,I_1}_n\setminus A^{I_0,I_1}_n)=3}$.}
  \label{Fig: t(I0,I1)}
  
\end{figure}

\begin{figure}[h]
\subcaptionbox*{}[.89
 \linewidth]{%
    \includegraphics[width=\linewidth]{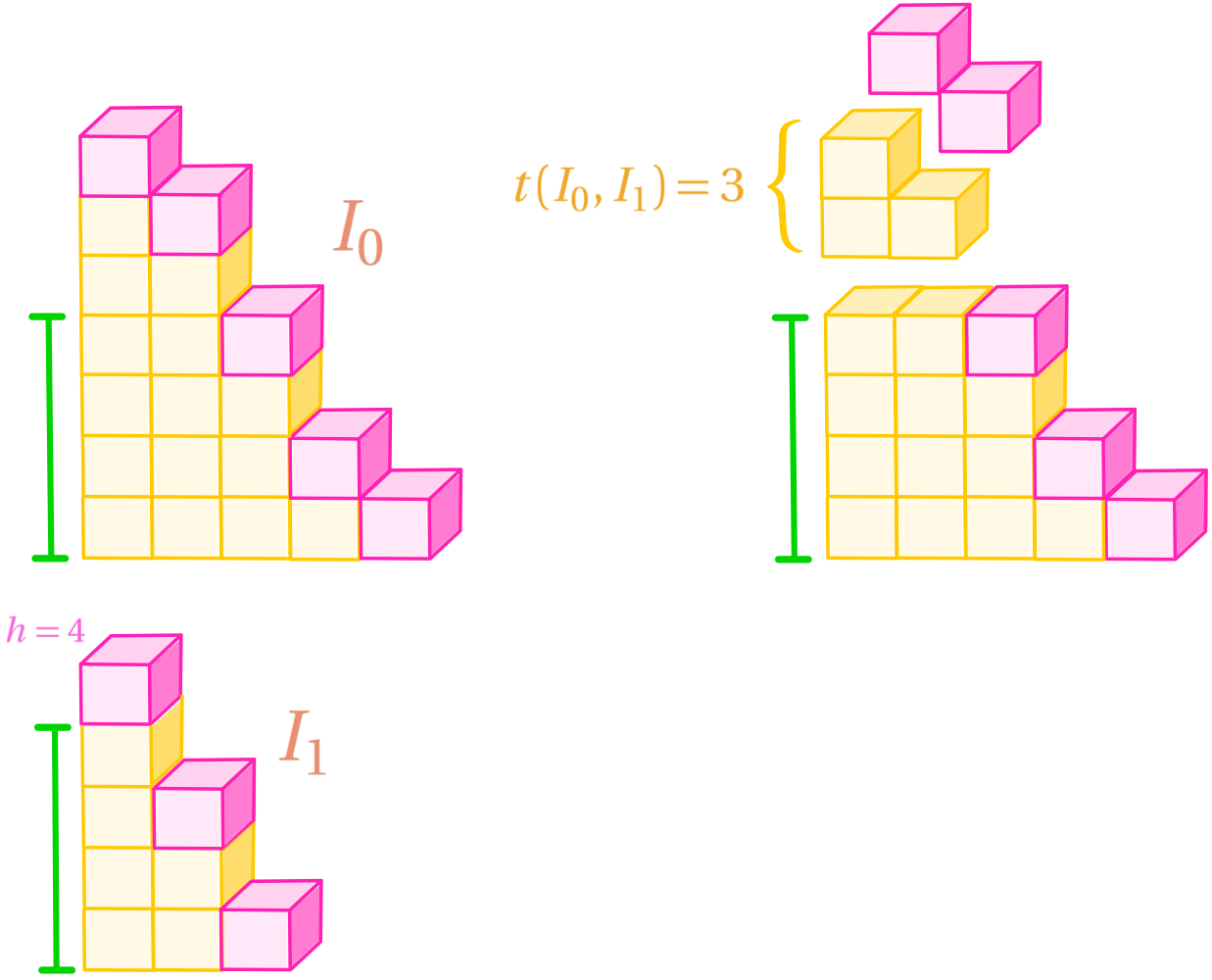}%
  }
  \vspace{-0.6cm}
  \caption{For $I_0=(y^4,z^6,y^3z,y^2z^3,yz^5)$ and $I_1=(y^2,z^4,yz^2)$, we have $h=4$, $t(I_0,I_1)=3$, and $\mathrm{Card(B^{I_0,I_1}_n\setminus A^{I_0,I_1}_n)=4}$.}
  \label{Fig: t(I0,I1)'}
  
\end{figure}
\end{example}

\begin{example}
    For $\Fm^k=\oplus_{i=0}^{m_1} x^iI_i$, one can easily check that $t(I_i,I_j)=\mathrm{Card}(B^{I_i,I_j}_n\setminus A^{I_i,I_j}_n)=\binom{j-i+1}{2}$. 
\end{example}

\section{The proof}\label{Section: proofs} In this section, we prove the main results.

First, in order to prove Theorem \ref{thm:upperbound}, we need a preparatory lemma and its immediate corollary concerning a lower bound for $t(I_i,I_j)$.
\begin{lemma}\label{Lem:bound}
    Let $I$ be a $0$-dimensional Borel-fixed ideal in $\BC[x,y,z]$, and let $\alpha=(\alpha_x,\alpha_y,\alpha_z)$ be a vector in $\BZ^3$ such that $\alpha_x<0$ and $\alpha_y,\alpha_z\geq0$ with $\alpha_x+\alpha_y+\alpha_z\leq-1$. Then, for any minimal generator $\gamma$ in $\tilde{I}$, we have $\gamma+\alpha\notin\tilde{I}$.
\end{lemma}
\begin{proof}
    If $\gamma+\alpha\in\tilde{I}$, then by Borel-fixedness and that $\alpha_x<0$, we deduce that $\gamma-\mathbf{e}_x\in\tilde{I}$, which contradicts the fact that $\gamma$ corresponds to a generator of $I$. 
\end{proof}

\begin{corollary}\label{cor:tlowerbound}
    For $I$ a 0-dimensional Borel-fixed ideal in $\BC[x,y,z]$, with $I=\bigoplus_ix^iI_i$ then
    \[\binom{j-i+1}{2} \leq t(I_i,I_j).\]
\end{corollary}

We now prove Theorem \ref{thm:upperbound}, following a similar argument as in \cite{Ramkumar-Sammartano}, but using our slightly improved inequalities to get a stronger bound which we can use to deduce, for example, the Briançon-Iarrobino Conjecture for $N=3$.
\begin{proof}[Proof of Theorem \ref{thm:upperbound}]
Recall that we have the decomposition $I=\bigoplus_ix^iI_i$. Any element of the tangent space $\mathrm{Hom}(I,R/I)$ can restrict to a non-zero element of $\mathrm{Hom}(I_i,S/I_j)$ for some $i,j$, and for basis elements, the restriction determines the original map. 

\emph{Claim.} Let $i>j$. Then, maps in $\mathrm{Hom}(I_i,S/I_j)$ with $\alpha$ such that $\alpha_y,\alpha_z\geq0$ and $\alpha_y+\alpha_z\leq i-j-2$ correspond to a single bounded connected component, but cannot extend to elements of the tangent space $\Hom(I,R/I)$, i.e., they are zero vectors.

\begin{proof}[Proof of Claim]
    Let $\alpha_x=j-i$, then by Lemma \ref{Lem:bound} for any minimal generator $\gamma\in\tilde{I}$ of $I$, we have
   \begin{align}
\gamma+\alpha+\mathbf{e}_y \notin \tilde{I},\label{1}\\
\gamma+\alpha+\mathbf{e}_x\notin\tilde{I}\label{2}.
   \end{align}
   
   Therefore, \eqref{1} implies that for generators $\gamma$ and $\gamma'$ whose $y$-coordinates differ by $1$, we have that $\gamma+\alpha$ and $\gamma'+\alpha$ are in the same component; hence, we have a single connected component in $(I_i+(\alpha_y,\alpha_z))\setminus I_j$. Also, we note that since $\alpha_y,\alpha_z\geq0$, the component is bounded.

    Similarly, \eqref{2} implies that there is only one connected component in $\tilde{I}_\alpha$, since for generators $\gamma$ and $\gamma'$ whose $z$-coordinates differ by 1, we have that  $\gamma+\alpha$ and $\gamma'+\alpha$ are in the same component. However, since $\alpha_x<0$, the component is unbounded.
\end{proof}
For fixed $i,j$, the number of maps in the claim above is $\binom{i-j}{2}$.

Also, we have 

\begin{align}\label{eq:doublebinomial}
    \sum_{j=0}^{m_1-1}\Bigg(\sum_{i=j+1}^{m_1}\binom{i-j}{2}\Bigg)= \sum_{j=0}^{m_1-1}{m_1-j+1\choose 3}={m_1+2 \choose 4}.
\end{align}

Thus, by the claim above, Corollary \ref{cor:ineq}, Corollary  \ref{cor:tlowerbound} and using \eqref{eq:doublebinomial} twice, it is straightforward to show
\begin{align*}
    T(I)&=\mathrm{hom}(I,R/I)\\
    &\leq\sum_{i=0}^{m_1}\sum_{j=0}^{m_1-1}\mathrm{hom}(I_i,S/I_j)-\sum_{j=0}^{m_1-1}\sum_{i=j+1}^{m_1}\binom{i-j}{2}\\
    &\leq\sum_{i=0}^{m_1}\sum_{j=0}^{m_1-1}\big(l_i+l_j-t(I_i,I_j)\big)-\binom{m_1+2}{4}\\
    &=\sum_{i=0}^{m_1}m_1l_i+\sum_{j=0}^{m_1-1}(m_1+1)l_j-\sum_{i=0}^{m_1}\sum_{j=0}^{m_1-1}t(I_i,I_j)-\binom{m_1+2}{4}\\
    &\leq(2m_1+1)l-\sum_{i=0}^{m_1}\sum_{j=i+1}^{m_1-1}\binom{j-i+1}{2}-\binom{m_1+2}{4}\\
    &=(2m_1+1)l-2\binom{m_1+2}{4},
\end{align*}
as required.
\end{proof}

Before proving Theorem \ref{thm:IncreasingU}, we need the following lemma.

\begin{lemma}\label{lem: m1Leqk}
    For a Borel-fixed ideal $I=(x^{m_1},y^{m_2},z^{m_3},\text{mixed terms})$  in $\BC[x,y,z]$ with  ${k+2 \choose 3}\leq l(I) <{k+3 \choose 3}$, we have $m_1\leq k$. 
\end{lemma}
\begin{proof}
Suppose that $m_1>k$. Then, since $I$ is Borel-fixed, we have $m_3\geq m_2 \geq m_1>k$. Therefore, each $m_i$ is at least $k+1$; hence, again using Borel-fixedness,
$\BN^3\setminus \tilde{I}$ contains at least a tetrahedron of size $k+1$. Hence,
the colength of $I$ will be at least ${k+3 \choose 3}$; contradiction.
\end{proof}
\begin{proof}[Proof of Theorem \ref{thm:IncreasingU}] For $l={k+2\choose 3}+\Delta$ with $0 \leq \Delta \leq {k+2\choose 2}-1$, we can write
\begin{align*}
    \psi'(m_1)=&\frac{\partial}{\partial m_1}\Bigg(\frac{(4m_1+2)(k^3+3k^2+2k+6\Delta)-m_1(m_1+2)(m_1^2-1)}{12}\Bigg)\\
    =&\frac{1}{6}(2k^3+6k^2+4k-2m_1^3-3m_1^2+m_1+1+12\Delta)>0.
\end{align*}

   The positivity is due to the inequality $m_1\leq k$ (Lemma \ref{lem: m1Leqk}).
\end{proof}

\begin{proof}[Proof of Corollary \ref{Cor:BIConj}] For $l=\binom{k+2}{3}$ we have
\begin{align*}
  \psi(k)=&(2k+1){k+2\choose 3}-2{k+2 \choose 4} = \frac{(2k+1)(k+2)(k+1)k}{6}-\frac{(k+2)(k+1)k(k-1)}{12}\\=&\frac{(k+2)(k+1)^2k}{4}={k+2\choose 2}{k+1\choose 2}=T(\mathfrak{m}^k).
\end{align*}
The last equality is obtained from \cite[Proposition III.4]{Bri-Iar} or \cite[Corollary 1.9]{Rezaee-23-Conjectures}. 
Combining this with Theorem  \ref{thm:IncreasingU}
implies the desired result.
   
\end{proof}

\begin{proof}[Proof of Corollary \ref{Cor:ConjB}] 
Since for $\mathfrak{m}^k$ we have $m_1=k$, and since by \cite[Lemma 1.7]{Rezaee-23-Conjectures}, $\mathfrak{m}^k$ is the only Borel-fixed ideal of colength ${k+2 \choose 3}$ with such a property, we get the desired result.
\end{proof}

\bibliographystyle{amsplain-nodash}

\bibliography{bib}

\vspace{0.25cm}

\end{document}